\documentclass[a4paper]{amsart}

\usepackage{verbatim,amsmath,amssymb,enumitem,amsthm}
\usepackage[displaymath,pagewise]{lineno} 
\usepackage{mathtools,xcolor}

\newtheorem{thm}{Theorem}

\newtheorem{lem}[thm]{Lemma}
\newtheorem{cor}[thm]{Corollary}

\newtheorem{rem}[thm]{Remark}

\theoremstyle{remark}

\theoremstyle{plain} 
\newcommand{\thistheoremname}{}
\newtheorem*{genericthm*}{\thistheoremname}
\newenvironment{namedthm*}[1]
{\renewcommand{\thistheoremname}{#1}%
	\begin{genericthm*}}
	{\end{genericthm*}}

\numberwithin{equation}{section} \numberwithin{thm}{section}

\newtheorem{definition}[thm]{Definition}

\newcommand*\patchAmsMathEnvironmentForLineno[1]{%
	\expandafter\let\csname old#1\expandafter\endcsname\csname #1\endcsname
	\expandafter\let\csname oldend#1\expandafter\endcsname\csname end#1\endcsname
	\renewenvironment{#1}%
	{\linenomath\csname old#1\endcsname}%
	{\csname oldend#1\endcsname\endlinenomath}}%
\newcommand*\patchBothAmsMathEnvironmentsForLineno[1]{%
	\patchAmsMathEnvironmentForLineno{#1}%
	\patchAmsMathEnvironmentForLineno{#1*}}%
\AtBeginDocument{%
	\patchBothAmsMathEnvironmentsForLineno{equation}%
	\patchBothAmsMathEnvironmentsForLineno{align}%
	\patchBothAmsMathEnvironmentsForLineno{flalign}%
	\patchBothAmsMathEnvironmentsForLineno{alignat}%
	\patchBothAmsMathEnvironmentsForLineno{gather}%
	\patchBothAmsMathEnvironmentsForLineno{multline}%
}

\newcommand{\dist}{{\rm dist}\,}
\newcommand{\rd}{{\mathbb R}^d}

\newcommand{\R}{{\mathbb R}}

\newcommand{\ep}{\varepsilon}

\newcommand{\K}{\mathcal{K}}

\newcommand{\eps}{\varepsilon}

\newcommand{\inter}{\mathrm{int}}

\def\en{\mathbb N}
\def\er{\mathbb R}
\def\q{\mathbb Q}
\def\qp{\mathbb Q^*}

\def\K{\mathcal K}

\def\S{\mathfrak S}

\newcommand{\graph}{\operatorname{graph}}

\newcommand{\conv}{\operatorname{conv}}

\newcommand{\WDC}{{W\! DC}}

\def\halfsq{\hbox{\kern1pt\vrule height 7pt\vrule width6pt height 0.4pt depth0pt\kern1pt}}
\def\ihalfsq{\hbox{\kern1pt \vrule width6pt height 0.4pt depth0pt
		\vrule height 7pt \kern1pt}}

\begin{document}

\title{Remarks on WDC sets}
\author{Du\v san Pokorn\'y}\author{Lud\v ek Zaj\'i\v cek}

\thanks{The research was supported by GA\v CR~18-11058S}

\begin{abstract} 
We study WDC sets, which form a substantial generalization of sets
	with positive reach and still admit the definition of curvature measures. 
	Main results concern WDC sets $A\subset \R^2$. We prove that, for such $A$, the distance function
	 $d_A= \dist(\cdot,A)$ is a ``DC aura'' for $A$, which implies that each locally WDC set in $\R^2$
	 is a WDC set. An another consequence is
	that compact WDC subsets of $\R^2$ form a Borel subset of the space
	of all compact sets.
\end{abstract}

\email{dpokorny@karlin.mff.cuni.cz}
\email{zajicek@karlin.mff.cuni.cz}

\date{\today}
\maketitle

\section{Introduction}
In \cite{PR} (cf. also \cite{FPR}, \cite{Fu2} and \cite{PRZ}), the authors introduced the class of WDC sets which form a substantial generalization of sets
 with positive reach and still admit the definition of curvature measures. 
The following  question naturally arises (see  \cite[Question 2, p. 829]{FPR} and \cite[10.4.3]{Fu2}).

\smallskip
{\bf Question.} 
  Is the distance function  $d_A= \dist(\cdot,A)$ of  each WDC set  $A\subset \R^d$ a DC aura for $F$ (see Definition~\ref{def:WDC})?
\smallskip

We answer this question positively
 in the case $d=2$  (Theorem \ref{thm:distanceAura} below); it remains open for $d\geq 3$.
The proof is based on 
 a characterization (proved in \cite{PRZ})  of locally WDC sets in $\R^2$ and the main result
 of \cite{PZ} which asserts that
\begin{equation}\label{dcgr}
\text{$d_A$ a DC function if $A\subset \R^2$ is a graph of a $DC$ function $g: \R \to \R$.}
\end{equation}
Recall that a function is called DC, if it is the difference of two convex functions and note
 that each set $A$ as in \eqref{dcgr} is a WDC set.

Theorem   \ref{thm:distanceAura} easily implies that each locally WDC set in $\R^2$ is WDC.
 Further, we use Theorem   \ref{thm:distanceAura} to prove that compact WDC subsets of $\R^2$  form a Borel subset of the space
 of all compact sets of $\R^2$ (Theorem  \ref{complwdc} (i)). The importance of this result is the fact that it 
 suggests that (at least in $\R^2$) a theory of point processes on the space of compact WDC sets (analogous to the concept of point processes on the space of sets with positive reach introduced in \cite{Zah86})
 can be build. 
 
 Concerning the compact WDC subsets of $\R^d$ for $d>2$, we are able to prove only a weaker fact that they form an analytic set
  (Theorem  \ref{complwdc} (ii)) which is not probably sufficient for the above mentioned application.

\section{Preliminaries}
\subsection{Basic definitions}\label{basic}
The symbol $\q$ denotes the set of all rational numbers. In any vector space $V$, we use the symbol $0$ for the zero element. 
We denote by $B(x,r)$ ($U(x,r)$) the closed (open) ball with centre $x$ and radius $r$.
 The boundary and the interior of a set $M$ are denoted by $\partial M$ and $\inter M$, respectively. A mapping is called $K$-Lipschitz if it is Lipschitz with a (not necessarily minimal) constant $K\geq 0$.

 The metric space of all real-valued continuous functions on a compact $K$ (equipped with the usual supremum metric $\rho_{\sup}$) will be denoted $C(K)$.

In the Euclidean space $\R^d$,  the norm is denoted by $|\cdot|$ and  the scalar product  by $\langle \cdot,\cdot\rangle$. By $S^{d-1}$ we denote the unit sphere in $\R^d$.


The distance function from a set $A\subset \R^d$ is $d_A\coloneqq  \dist(\cdot,A)$ and 
 the metric projection of $z\in \R^d$ to $A$ is   
 $\Pi_A(z)\coloneqq \{ a\in A:\, \dist(z,A)=|z-a|\}$.


\subsection{DC functions}\label{dc}

 Let $f$ be a real function defined on an open convex set $C \subset \R^d$. Then we say
	 that $f$ is a {\it DC function}, if it is the difference of two convex functions. Special DC
	 functions are semiconvex and semiconcave functions. Namely, $f$ is a {\it semiconvex} (resp.
	 {\it semiconcave}) function, if  there exist $a>0$ and a convex function $g$ on $C$ such that
	$$   f(x)= g(x)- a \|x\|^2\ \ \ (\text{resp.}\ \  f(x)= a \|x\|^2 - g(x)),\ \ \ x \in C.$$

We will use the following well-known properties of DC functions.

\begin{lem}\label{vldc}
Let $C$ be an open convex subset of $\R^d$. Then the following assertions hold.
\begin{enumerate}
\item[(i)]
	If $f: C\to \R$ and $g: C\to \R$ are DC, then (for each $a\in \R$, $b\in \R$) the functions $|f|$, $af + bg$, 
	 $\max(f,g)$ and $\min(f,g)$ are DC.
	\item[(ii)]
	Each locally DC  function  $f:C \to \R$  is DC. 
		\item[(iii)]	Each  DC  function  $f:C \to \R$  is 
  Lipschitz on each compact convex set $Z\subset C$.
	\item[(iv)]
Let $f_i: C \to \R$, $i=1,\dots,m$, be DC functions. Let $f: C \to \R$ be a continuous function
 such that $f(x) \in \{f_1(x),\dots,f_m(x)\}$ for each $x \in C$. Then $f$ is DC on $C$.
\end{enumerate}
\end{lem}
\begin{proof}
Property (i) follows easily from definitions, see e.g. \cite[p. 84]{Tuy}. Property (ii) was proved in \cite{H}.
 Property (iii) easily follows from the local Lipschitzness of  convex functions. Assertion (iv) is a special case of \cite[Lemma 4.8.]{VeZa} (``Mixing lemma'').
\end{proof}


It is well-known (cf. \cite{PZ}) that if $\emptyset \neq A\subset \R^d$ is closed, then $d_A$ need not be DC; however (see, e.g., \cite[Proposition 2.2.2]{CS}), 
\begin{equation}\label{loksem}
\text{$d_A $ is locally semiconcave (and so locally DC) on $\R^d \setminus A$.}
\end{equation}

\subsection{Clarke generalized gradient}\label{clar}

If  $U\subset \R^d$ is an open set, $f:U\to\R$  is locally Lipschitz and $x\in U$, we denote by $\partial_C f(x)$ the {\it generalized gradient of $f$ at $x$}, which can be defined as the closed convex hull of all limits $\lim_{i\to\infty}f'(x_i)$ such that $x_i\to x$ and $f'(x_i)$ exists for all $i\in\en$ (see \cite[S1.1.2]{C}; $\partial_C f(x)$ is also called \emph{Clarke subdifferential of $f$ at $x$} in the literature). Since we identify $(\R^d)^*$ with $\R^d$ in the standard way, we sometimes consider $\partial_C f(x)$ as a subset of $\R^d$. We will repeatedly use the fact that the mapping $x\mapsto\partial_C f(x)$ is upper semicontinuous and, hence (see \cite[Theorem~2.1.5]{C}), 
\begin{equation}\label{uzcl}
v\in\partial_C f(x)\ \ \text{ whenever}\ \  x_i\to x,\  v_i\in\partial_C f(x_i)\ \ \text{ and}\ \  v_i\to v.
\end{equation}
We also use that $|u|\leq K$ whenever $u\in\partial_C f(x)$ and $f$ is $K$-Lipschitz on a neighbourhood
 of $x$.
Obviously,
\begin{equation}\label{cls}
\partial_C (\alpha f)(x) = \alpha \partial_C f(x).
\end{equation}

Recall that
\begin{equation}\label{clder1}
f^0(x,v)\coloneqq  \limsup_{y\to x, t\to 0+}  \ \frac{f(y+tv)-f(y)}{t}
\end{equation}
and  (see \cite{C})
\begin{equation}\label{clder2}
f^0(x,v) = \sup \{ \langle v, \nu\rangle:\ \nu\in \partial_Cf(x)\}.
\end{equation}
We will need the following simple lemma.
\begin{lem}\label{clod}
Let $f$ be a Lipschitz function on an open set $G \subset \R^d$, $x \in G$ and $\ep>0$.
\begin{enumerate}
\item[(i)] 
If $\dist(0, \partial_Cf(x)) \geq 2\ep$, then
\begin{equation}\label{derpod}
\exists v\in  S^{d-1}, \rho>0 \ \forall y \in U(x,\rho), 0<\alpha< \rho:\  \frac{f(y+\alpha v) -f(y)}{\alpha} \leq - \ep.
\end{equation}
\item[(ii)]
If \eqref{derpod} holds, then $\dist(0, \partial_Cf(x)) \geq \ep$.
\end{enumerate}
\end{lem} 
\begin{proof}
(i)\ Let $\dist(0, \partial_Cf(x)) \geq 2\ep$. Since $\partial_Cf(x)$ is convex, there exists
 (see e.g. \cite[Theorem 1.5.]{DL}) $v\in  S^{d-1}$
 such that
$$\dist(0, \partial_Cf(x))= - \sup\{\langle v, \nu\rangle:\ \nu \in \partial_Cf(x)\}.$$
So, by \eqref{clder2}, $f^0(x,v) \leq -2\ep$ and thus \eqref{clder1} implies \eqref{derpod}.

(ii)\ If  \eqref{derpod} holds, choose corresponding 
 $v\in  S^{d-1}$ and $ \rho>0$. Then  $f^0(x,v) \leq -\ep$ by \eqref{clder1}. Consequently, by
 \eqref{clder2},
$- |\nu| \leq \langle v, \nu\rangle \leq -\ep$ for each $\nu \in \partial_Cf(x)$ and so 
 $\dist(0, \partial_Cf(x)) \geq \ep$.
\end{proof}

\subsection{ WDC sets}\label{wdc}

WDC sets (see the definition below) which provide a natural generalization of sets with positive reach were defined in
 \cite{PR} using Fu's notion of an ``aura'' of a set (see, e.g., \cite{Fu2} for more information).
 Note that the notion of a DC aura were defined in \cite{PR} and \cite{FPR} by a formally different but equivalent way (cf. \cite[Remark 2.12 (v)]{PRZ}).

\begin{definition}[cf. \cite{PRZ}, Definitions~2.8, 2.10]\label{def:WDC}
	Let $U\subset\R^d$ be open and $f:U\to\R$ be locally Lipschitz. A number $c\in \R$ is called a
 {\it weakly regular value} of $f$ if whenever $x_i\to x$, $f(x_i)>c=f(x)$ and $u_i\in\partial_C f(x_i)$ for all $i\in\en$ then $\liminf_i|u_i|>0$.

A set $A\subset\rd$ is called {\it WDC} if there exists a DC function $f:\rd\to [0,\infty)$ such that $A=f^{-1}\{0\}$ and $0$ is a weakly regular value of $f$. 
In such a case, we call $f$ a {\it DC aura} (for $A$).

A set $A\subset \R^d$ is called locally WDC if for any point $a \in A$ there exists a WDC set $A^*\subset \R^d$
 that agrees with $A$ on an open neighbourhood of $a$.
	\end{definition}
	
	(Note that a weakly regular value of $f$ need not be  in the range of $f$, and so $\emptyset$
	  is clearly a WDC set  by our definition.)
		\smallskip
		
		Note that a set $A \subset \R^d$ has a positive reach at each point if and only there exists a DC aura for $A$ which is even semiconvex (\cite{Ba}).

\section{Distance function of a WDC set in $\R^2$ is an aura}\label{distwdc}

First we present (slightly formally rewritten)  \cite[Definition~7.9]{PRZ}.

\begin{definition}\label{sektory}
	\begin{enumerate}
		\item[(i)]
		A set $S \subset \R^2$ will be called a {\it basic open DC sector} (of radius $r$) if 
		$S= U(0,r) \cap \{(u,v) \in \R^2:\ u\in (-\omega, \omega), v >f(u)\}$, where  $0<r<\omega$ and $f$ is a DC function on $(-\omega, \omega)$
		such that $f(0)=0$, $R(u)\coloneqq  \sqrt {u^2 + f^2(u)}$ is strictly increasing on $[0,\omega)$ and 
		strictly decreasing on $(-\omega,0]$. 
		
		By an {\it open DC sector} (of radius $r$) we mean an image $\gamma(S)$ of a  basic open DC sector $S$ (of radius $r$) under a rotation around the origin $\gamma$.
		
		\item[(ii)]
		A set of the form $\gamma( \{(u,v) \in \R^2:\ u\in [0, \omega),   g(u) \leq v \leq f(u)\}  ) \cap U(0,r)$, where $\gamma$ is a rotation around the origin, $0<r<\omega$ and $f,g: \R\to\er$ are DC functions such that $g \leq f$ on $[0,\omega)$, $f(0)=g(0)=f'_+(0)=g'_+(0)=0$ and
		the functions $R_f(u)\coloneqq  \sqrt {u^2 + f^2(u)}$, $R_g(u)\coloneqq  \sqrt {u^2 + g^2(u)}$ are strictly increasing on $[0,\omega)$,  will be called a
		{\it
			degenerated closed DC sector} (of radius $r$).
	\end{enumerate}
\end{definition}
We will use the following complete characterization of locally WDC sets in $\R^2$ 
 (\cite[Theorem~8.14]{PRZ}).
\begin{namedthm*}{Theorem PRZ}\label{T:characterisation} 
	Let $M$ be a closed subset of $\er^2$. Then $M$ is a  locally WDC set if and only if for each $x\in\partial M$
	there is $\rho>0$ such that one of the following conditions holds:
	\begin{enumerate}[label={\rm (\roman*)}]
		\item\label{cond:singlePointNew} $M\cap U(x,\rho)=\{x\}$,
		\item\label{cond:oneConeNew} there is a degenerated closed DC sector $C$ of radius $\rho$ such that $$M\cap U(x,\rho)=x+C,$$
		\item\label{cond:manyConesNew} there are pairwise disjoint  open DC sectors $C_1,\dots,C_k$ of radius $\rho$ such that 
		\begin{equation}\label{kolac}
		U(x,\rho) \setminus M =\bigcup_{i=1}^k \left(x+C_i\right).
		\end{equation}
	\end{enumerate}
\end{namedthm*}

\begin{lem}\label{lem:distranceFromGraphWR}
	Let $f$ be an $L$-Lipschitz function on $\er$. 
	Denote $d\coloneqq \dist(\cdot,\graph f)$.
	Then $|\xi_2|\geq \frac{1}{\sqrt{L^2+1}}$ whenever
	 $\xi=(\xi_1,\xi_2)\in\partial_C d(x)$ and
	$x\in\er^2\setminus\graph f$.
\end{lem}
\begin{proof}
	Pick $x\in\er^2\setminus\graph f$.
	Without any loss of generality we can assume that $x=0$. We will assume that $f(0)<0$; the case
	 $f(0)>0$ is quite analogous.
	Denote $r\coloneqq d(0)$ and $P\coloneqq  \Pi_{\graph f}(0)$. 
	 Set  $g(u)\coloneqq-\sqrt{r^2-u^2}$, $u\in[-r,r]$.
	Clearly  $f\leq g$ on $[-r,r]$ and 
	 $(u,v)\in P$ if and only if $f(u)=g(u)=v$.
	We will show that 
	\begin{equation}\label{odhu}
	\text{$|u|\leq \frac{Lr}{\sqrt{1+L^2}}$ whenever $(u,v)\in P$.}
	\end{equation}
	To this end, suppose $(u,v)\in P$. If $u>0$, then
	$$ L\geq \frac{f(t)-f(u)}{t-u} \geq \frac{g(t)-g(u)}{t-u}\ \ \text{for each}\ \ 0<t<u,$$
	 and consequently $L\geq g'_-(u)$.  Therefore $u<r$ and  $L \geq u (r^2-u^2)^{-1/2}$.
	 Analogously considering $g'_+(u)$, we obtain for $u<0$ that $u>-r$ and  $u (r^2-u^2)^{-1/2} \geq -L$. In both cases we have $L \geq |u| (r^2-u^2)^{-1/2}$ and an elementary computation gives
	 \eqref{odhu}.

	Using \eqref{odhu} we obtain that if $(u,v)\in P$ then
	\begin{equation}\label{eq:estimateg(u)}
	v=g(u)\leq -\sqrt{r^2-\left(\frac{Lr}{\sqrt{1+L^2}}\right)^2}
	=-\frac{r}{\sqrt{1+L^2}}.
	\end{equation}
	By \cite[Lemma~4.2]{Fu} and \eqref{eq:estimateg(u)} we obtain  
	\begin{equation*}
	\partial_C d(0)=\conv \left\{ \frac{1}{r}(-u,-v): (u,v)\in P \right\}\subset  \er \times \left[\frac{1}{\sqrt{L^2+1}},\infty\right)
	\end{equation*}
	and the assertion of the lemma follows.
\end{proof}

\begin{thm}\label{thm:distanceAura}
	Let $M\ne \emptyset$ be a  locally $\WDC$ set in $\er^2$. Then the distance function $d_M$ is a DC aura for $M$.
	 In particular, $M$ is a WDC set.
\end{thm}
\begin{proof}
	Denote $d\coloneqq  d_M$.
	For each  $x\in\partial M$  choose $\rho= \rho(x)$ by Theorem~PRZ.
	We will prove that
	\begin{enumerate}[label={\rm (\alph*)}]
		\item\label{cond:DC} $d$ is DC on $U\left( x,\frac{\rho}{3} \right)$,
		\item\label{cond:weaklyRegular} there is $\eps= \eps(x)>0$ such that $|\xi|\geq\eps$ whenever $y\in U\left( 0,\frac{\rho}{3} \right)\setminus M$ and  $\xi\in\partial_C d(y)$.
	\end{enumerate}
	Without any loss of generality we can assume that $x=0$.

If case \ref{cond:singlePointNew} from Theorem~PRZ holds, then $d(y)=|y|$, $y\in U\left( 0,\frac{\rho}{3} \right)$, and so $d$ is convex and therefore DC on $U\left(0,\frac{\rho}{3} \right)$. Similarly, condition \ref{cond:weaklyRegular} holds as well, since if 
 $y\in U\left( 0,\frac{\rho}{3} \right)\setminus M$ and $\xi\in\partial_C d(y)$ then $\xi=\frac{y}{|y|}$ and so $|\xi|=1$.

If case \ref{cond:oneConeNew} from Theorem~PRZ holds,
we know that $M\cap U(0,\rho)$ is a degenerated closed DC sector $C$ of radius $\rho$.
Let $\gamma$, $f$, $g$ and $\omega$ be as in Definition~\ref{sektory}.
Without any loss of generality we may assume that $\gamma$ is the identity map.

By Lemma \ref{vldc} (iii) we can choose $L>0$ such that both $f$ and $g$ are $L$-Lipschitz
 on $[0,\rho]$
 and define
\begin{equation*}
\tilde f(u)\coloneqq
\begin{cases}
f(u)& \text{if}\quad 0\leq u \leq \rho,\\
f(\rho)& \text{if}\quad \rho<u,\\
2Ly& \text{if}\quad u<0.
\end{cases}
\ \ \ \ 
\text{and} 
\ \ \ \ 
\tilde g(u)\coloneqq
\begin{cases}
g(u)& \text{if}\quad 0\leq u \leq \rho,\\
g(\rho)& \text{if}\quad \rho<u,\\
-2Lu& \text{if}\quad u<0.
\end{cases}
\end{equation*}

It is easy to see that both $\tilde f$ and $\tilde g$ are $2L$-Lipschitz and they
 are DC by Lemma \ref{vldc} (iv).

Put 
\begin{equation*}
 M_0\coloneqq \left\{ (u,v)\in\er^2: u\geq0, \; \tilde g(u)\leq v\leq\tilde f(u)\right\},
\end{equation*}
\begin{equation*}
M_1\coloneqq \left\{ (u,v)\in\er^2: u\geq0, \; \tilde f(u)< v\right\}\cup
 \left\{ (u,v)\in\er^2: u<0, \; -\frac{u}{2L}< v\right\},
\end{equation*}
\begin{equation*}
M_2\coloneqq \left\{ (u,v)\in\er^2: u\geq0, \; \tilde g(u)> v\right\}\cup
\left\{ (u,v)\in\er^2: u<0, \; \frac{u}{2L}> v\right\}
\end{equation*}
and
\begin{equation*}
M_3\coloneqq 
\left\{ (u,v)\in\er^2:\frac uL< v< -\frac uL\right\}.
\end{equation*}
Clearly $\er^2= M_0\cup M_1\cup M_2\cup M_3$ and  $M_1$, $M_2$, $M_3$ are open.
 
Set  $\tilde d\coloneqq\dist(\cdot, M_0)$ and, for each $y \in \R^2$, define
 $$d_0(y)=0,\ \  d_1(y)\coloneqq \dist(y,\graph \tilde f),\ \
d_2(y)\coloneqq \dist(y,\graph \tilde g),
\ \ d_3(y)\coloneqq |y|.$$
Functions $d_1$ and $d_2$ are DC on $\er^2$ by \eqref{dcgr}, $d_0$ and $d_3$ are convex and therefore DC on $\er^2$.

Using (for $K= 1/L, -1/L, 1/(2L), -1/(2L)$) the facts that the lines with the slopes $K$ and $-1/K$
 are orthogonal and $M_0\subset \{(u,v): u\geq 0,\; -Lu\leq v \leq Lu \}$, 
 easy geometrical observations show that
\begin{equation}\label{mjdt}
\tilde d(y)= d_i(y)\ \ \text{ if}\ \  y \in M_i,\  0\leq i \leq 3,
\end{equation}
 and so Lemma \ref{vldc} (iv)  implies that $\tilde d$ is DC.

Now pick an arbitrary $y\in \er^2\setminus M_0=M_1\cup M_2\cup M_3$ and choose $\xi=(\xi_1,\xi_2)\in\partial_C \tilde d(y)$. Using \eqref{mjdt}, we obtain that
if $y\in M_3$ then $\xi=\frac{y}{|y|}$ and so $|\xi|=1$. Using Lemma 
\ref{lem:distranceFromGraphWR}, we obtain that 
if $y\in M_1\cup M_2$, then $|\xi|\geq |\xi_1|\geq \frac{1}{\sqrt{4L^2+1}}$.

Now, since $d=\tilde d$ on $U\left( 0,\frac{\rho}{3} \right)$ both \ref{cond:DC} and \ref{cond:weaklyRegular} follow.

It remains to prove \ref{cond:DC} and \ref{cond:weaklyRegular} if case \ref{cond:manyConesNew} 
from Theorem~PRZ holds.
Let $C_i$, $i=1,\dots, k$, be the open DC sectors as in \ref{cond:manyConesNew}.
Denote $A_i\coloneqq \er^2\setminus C_i$
and define $\delta_i\coloneqq \dist(\cdot,A_i)$, $i=1,\dots, k$.

Note that, for $y\in U\left( 0,\frac{\rho}{3} \right)$, one has
\begin{equation*}
d(y)=
\begin{cases}
\delta_i(y)& \text{if}\; y\in C_i,\\
0&\text{if}\; y\in M.
\end{cases}
\end{equation*}
Therefore (by Lemma~\ref{vldc} (iv)) it is enough to prove that 
\ref{cond:DC} and \ref{cond:weaklyRegular}
hold with $d$ and $M$ being replaced by $\delta_i$ and $A_i$, respectively ($i=1,\dots, k$). 
Fix some $i\in \{1,\dots, k\}$.
Without any loss of generality we can assume that $C_i$ is an basic open DC sector of radius $\rho$ with corresponding $f_i$ and $\omega_i$.
Now define
\begin{equation*}
\tilde f_i(u)\coloneqq
\begin{cases}
f_i(u)&\text{if}\; u\in[-\rho,\rho],\\
f_i(-\rho)&\text{if}\; u<-\rho,\\
f_i(\rho)&\text{if}\; u>\rho.
\end{cases}
\end{equation*}
Then $\tilde f_i$ is Lipschitz and DC on $\er$.
Put $\tilde d_i(y)=\dist(y,\graph \tilde f_i)$.
Then $\tilde d_i$ is DC by \eqref{dcgr} and $0$ is a weakly regular value of $\tilde d_i$ by Lemma~\ref{lem:distranceFromGraphWR}.
And since $d_i=\tilde d_i$ on $U\left( 0,\frac{\rho}{3} \right)$ we are done.

Since $d$ is locally DC on $\er^2\setminus M$ by \eqref{loksem} and on the interior of $M$ (trivially), (a) implies that $d$ is locally DC and so DC by Lemma \ref{vldc} (ii). Further, (b)
 immediately implies that $0$ is a weakly regular value of $d$ and thus $d=d_M$ is an aura for $M$.
\end{proof}

\begin{rem}
By	Theorem~\ref{thm:distanceAura}, in $\er^2$  locally $\WDC$ sets and WDC sets coincide. 
This gives a partial answer to the part of \cite[Problem 10.2]{PR} which asks whether
 the same is true in each $\R^d$. 
\end{rem}

\section{Complexity of the system of WDC sets}\label{compl}

	In the following,  we will work in each moment in an $\R^d$ with a fixed $d$, and so for simplicity we will use the notation, in which the dependence on $d$ is usually omitted.
	
	The space of all nonempty compact subsets of $\R^d$ equipped with the usual Hausdorff metric
	$\rho_H$ is denoted by $\K$. It is well-known (see, e.g., \cite[Proposition~2.4.15 and Corollary~2.4.16]{Sri98}) that $\K$ is a separable
	 complete metric space. For a closed set $M\subset \er^d$, we set
	$\K(M)\coloneqq  \{K \in \K: K \subset M\}$, which is clearly a closed subspace of $\K$.
	 The set of all nonempty compact $\WDC$ subsets of $M\subset \er^d$ will be denoted by $\WDC(M)$.

In this section, we will prove the following theorem.
\begin{thm}\label{complwdc}
\begin{enumerate}
\item[(i)]\ $WDC(\R^2)$\ is an $F_{\sigma\delta\sigma}$ subset of $\K(\R^2)$.
\item[(ii)]\  $WDC(\R^d)$\ is an analytic subset of $\K(\R^d)$ for each $d \in \en$.
\end{enumerate}
\end{thm}
Before the proof of this theorem, we introduce some spaces,
  make a number of observations, and prove a technical lemma.

	First observe that  $WDC(\R^d)=  \bigcup_{n=1}^{\infty}  WDC(B(0,n))$ and so, to prove Theorem \ref{complwdc}, it is sufficient to prove that, for each $r>0$,
\begin{equation}
	\begin{aligned}\label{staci2}
	&\text{for $d=2$ (resp.  $d \in \en$), $WDC(B(0,r))$}\\ 
	&\text{is an $F_{\sigma\delta\sigma}$ (resp. analytic) subset of $\K(B(0,r))$.}
	\end{aligned}
\end{equation}
						
Further observe that it is sufficient to prove 	\eqref{staci2} for $r=1$. Indeed, denoting
 $H(x)\coloneqq  x/r,\ x \in \R^d$, it is obvious that $H^*: K \mapsto H(K)$ gives a homeomorphism
 of $\K(B(0,r))$ onto $\K(B(0,1))$ and $H^*(WDC(B(0,r))=WDC(B(0,1))$ (clearly $f$ is an aura for
 $K$ if and only if $f \circ H^{-1}$ is an aura for $H^*(K)$).

To prove \eqref{staci2} for $r=1$, we will consider the space
 $X$  of all $1$-Lipschitz functions $f:B(0,4)\to [0,4]$
such that $f\geq 1$ on $B(0,4) \setminus U(0,3)$, equipped with the supremum metric $\rho_{\sup}$. 
Obviously, $X$ is a closed subspace of $C(B(0,4))$ and so it is a separable complete metric space.

The motivation for introducing $X$ is the fact that
\begin{equation}\label{moti}
\text{if $K\in \K(B(0,1))$, then  $f_K\coloneqq  d_K\restriction_{B(0,4)} \in X$.}
\end{equation}
Since we are interested in $K \in WDC(B(0,1))$, we define also two subspaces of $X$:
$$ A\coloneqq  \{f\in X:\ 0\ \ \text{ is a weakly regular value of }\ f|_{U(0,4)}\},$$
$$D\coloneqq  \{f\in X:\ f=g-h\ \ \text{for some convex Lipschitz functions}\ \ g, h\ \ \text{on}\ \ B(0,4)\}.$$
Their complexity is closely related to the complexity of $WDC(B(0,1))$, as the following  lemma 
 indicate.

\begin{lem}\label{lem:auraINXd}
Let $\emptyset \neq K\subset B(0,1)\subset\er^d$ be compact. Then:
\begin{enumerate}
\item[(i)]
K is WDC if and only if there is a function $g\in D\cap A$ such that $K=g^{-1}(0)$.
\item[(ii)]
If $d=2$, then K is WDC if and only if $f_K\coloneqq  d_K\restriction_{B(0,4)} \in D \cap A$.
\end{enumerate}
\end{lem}
\begin{proof}
	(i)\ Suppose first that $K$ is WDC and  $f$ is an aura for $K$.
Using Lemma \ref{vldc}(iii), we can chose $\alpha>0$ so small that  the function $\alpha f$ is
 $1$-Lipschitz on $B(0,4)$ and  $0\leq \alpha f(x) \leq 4$ for $x \in B(0,4)$. Set
$$  h(x)\coloneqq 	\max(|x|-2, \alpha f(x)), \ x \in \R^d,\ \ \text{and}\ \ g\coloneqq  f\restriction_{B(0,4)}.$$
Then clearly $K=g^{-1}(0)$. Since
 $h$ is DC on $\R^d$ by  Lemma \ref{vldc}(i), we obtain $g \in D$ by Lemma \ref{vldc}(iii).
 Finally, $g \in A$, since $g=\alpha f$ on $U(0,2)$.

	Conversely, suppose  that $K=g^{-1}(0)$ for some  $g\in A\cap D$ and set
	\begin{equation*}
	f(x)\coloneqq
	\begin{cases}
	\min(g(x), 1), & \mbox{if } x\in U(0,4),\\
	1 & \mbox{otherwise}.
	\end{cases}
	\end{equation*}
	Since $f$ is DC on $U(0,4)$ by Lemma \ref{vldc} (i) and $f=1$ on $\R^d \setminus B(0,3)$, we see that
	$f$ is locally DC and so DC by  Lemma  \ref{vldc} (ii). 
	Since $0$ is clearly a weakly regular value of $f$, we obtain that $f$ is an aura for $K$.
	
	(ii)\ If  $K$ is WDC, first note that $f_K \in X$ (see \eqref{moti}). Since $d_K$
	 is an aura for $K$ by  Theorem \ref{thm:distanceAura}, we obtain immediately that $f_K \in A$, and also $f_K \in D$
	 by  Lemma  \ref{vldc} (iii). 
	
	If $f_K \in A \cap D$, then $K$ is WDC by (i).
	\end{proof}
	
For the application of Lemma \ref{lem:auraINXd}(ii)   we need the simple fact that
\begin{equation}\label{spojpsi}
\text{$\Psi: K \mapsto f_K$, \ $K \in  \K(B(0,1))$, \ \text {is a continuous mapping into}\ \ $X$.}
\end{equation} 
Indeed, if $K_1, K_2 \in  \K(B(0,1))$ with $\rho_H(K_1, K_2) <\ep$   and $x \in B(0,4)$, then clearly
$ d_{K_1}(x) <  d_{K_2}(x) + \ep$, $ d_{K_2}(x) <  d_{K_1}(x) + \ep$, and consequently
$\rho_{\sup}(f_{K_1}, f_{K_2}) \leq \ep$.

Further observe that
\begin{equation}\label{bod}
 D\ \ \text{is an}\ \ F_{\sigma}\ \ \text{subset of}\ \ \ X.
\end{equation}
 To prove it, for each $n \in \en$ set
$$ C_n\coloneqq  \{g\in C(B(0,4)):\ g \ \text{is convex}\ n\text{-Lipschitz and}\ \ |g(x)| \leq 4n+4,\ x \in B(0,4)\}.$$
Now observe that if $f \in D$ then we can choose $n \in \en$ and convex $n$-Lipschitz functions $g$, $h$ 
 such that $f=g-h$, $g(0)=0$ and consequently $\|g\|\leq 4n$, $\|h\|\leq 4n+4$, and so $g,\, h \in C_n$.
Consequently,  $D = X \cap  \bigcup_{n=1}^{\infty} (C_n-C_n)$. Each $C_n$ is clearly closed in
 $C(B(0,4))$ and so it is
 compact in 
$C(B(0,4))$ by the Arzel\` a-Ascoli theorem. Consequently also $C_n-C_n = \sigma(C_n \times C_n)$, where
 $\sigma$ is the continuous mapping $\sigma: (g,h) \mapsto g-h$, is compact, and \eqref{bod} follows. 

The most technical part of the proof of Theorem \ref{complwdc} is to show that $A$ is an 
 $F_{\sigma\delta\sigma}$ subset of $X$. To prove it,  we need some lemmas.
\begin{lem}\label{swr2}
Let $f\in X$. Then $f\in A$ if and only 
\begin{equation}\label{wr2}
\exists\, 0<\eps\;
\forall x\in f^{-1}(0,\eps),\; \nu\in\partial_C f(x)\;
:|\nu|\geq \eps.
\end{equation}
\end{lem}
\begin{proof}
If \eqref{wr2} holds, then we easily obtain $f \in A$ directly from the definition of
a weakly regular value.

To prove the opposite implication, suppose that $f \in A$ and \eqref{wr2} does not hold.
 Then there exist points $x_n \in f^{-1} (0, 1/n),\ n\in \en,$ and $\nu_n\in\partial_C f(x_n)$
 such that $|\nu_n|  < 1/n$. Choose a subsequence $x_{n_k} \to x \in B(0,4)$. Since
$0\leq f(x_{n_k}) < 1/n_k$, we have $f(x_{n_k}) \to f(x)=0$, and consequently $x \in U(0,4)$.
 Since $\nu_{n_k} \to 0$, we obtain that $0$ is not a weakly regular value of $f|_{U(0,4)}$, which contradicts 
 $f \in A$.
\end{proof}

Denote $\qp\coloneqq \q\cap (0,1)$ and for every $\ep \in\qp$ and $d\in\en$ pick a finite set $\S^d_{\ep}\subset S^{d-1}$ such that for every $v\in S^{d-1}$ there is some $\nu\in\S^d_{\ep}$ satisfying $|v-\nu|< \ep$.

\begin{lem}\label{lem:A_dFormula}
	Let $f$ be a function from $X$. Then $f\in A$ if and only if
	\begin{equation}\label{eq:A_dFormula}
	\begin{aligned}
	\exists \eps \in\qp\;
	\forall p,q\in\qp, 0<p<q<\eps \;
	\exists \rho\in\qp \;
	\forall x\in U(0,4):\;\\
	(f(x)\notin (p,q) \vee 
	\exists\nu\in\S^d_{\eps}\,\forall y\in U(x,\rho), 0<\alpha<\rho: f(y+\alpha \nu)-f(y)\leq-\eps \alpha).
	\end{aligned}
	\end{equation}
\end{lem}

\begin{proof}
First suppose that \eqref{eq:A_dFormula} holds and choose $\eps \in \qp$ by \eqref{eq:A_dFormula}. We will show that
\begin{equation}\label{hv}
	\forall x\in f^{-1}(0,\eps),\; \nu\in\partial_C f(x)\;
	:|\nu|\geq \eps.
	\end{equation}
	To this end, consider an arbitrary  $x\in f^{-1}(0,\eps)$ and choose $p,q\in\qp$ such that $0<p<q<\eps$ and
	 $ f(x) \in (p,q)$. Choose $\rho \in \qp$ which exists for $\eps, p, q$ by  \eqref{eq:A_dFormula}. So there exists
	 $ \nu\in\S^d_{\eps}$  such that
	 $$ \forall y\in U(x,\rho), 0<\alpha<\rho: f(y+\alpha v)-f(y)\leq-\eps \alpha.$$
	 Therefore Lemma \ref{clod} (ii) gives  that $|\nu|\geq \eps$ for each $\nu\in\partial_C f(x)$. Thus \eqref{hv}
	  holds and so  $f\in A$ by Lemma \ref{swr2}.
	  
	  Now suppose   $ f \in A$. Using \eqref{wr2}, we can choose $\eps \in\qp$ such that
	  \begin{equation}\label{4hv}
	\forall x\in f^{-1}(0,\eps),\; \nu\in\partial_C f(x)\;
	:|\nu|\geq 4 \eps.
	\end{equation}
	To prove \eqref{eq:A_dFormula}, consider arbitrary $ p,q\in\qp$, $0<p<q<\eps$. Using Lemma  \ref{clod} (i), we easily obtain that for each $z \in K\coloneqq  f^{-1}([p,q])$ there exist  $\rho(z)>0$ and  $v(z)\in S^{d-1}$ such that 
\begin{equation}\label{ct}	
\forall y\in U(z,\rho(z)), 0<\alpha<\rho(z): f(y+\alpha v(z))-f(y)\leq-2\eps \alpha.	
\end{equation}
 Choose $\rho \in \qp$ as a Lebesgue number (see \cite{E}) of the open covering  
 $\{U(z,\rho(z))\}_{z\in K}$  of the compact $K$.
 For an arbitrary $x \in U(0,4)$, either $f(x)\notin (p,q)$ or $x \in K$. In the second case, by the definition of Lebesgue number, there exists  $z\in K$ such that  $U(x,\rho)\subset U(z,\rho(z))$. Then clearly $\rho < \rho(z)$ and so \eqref{ct} implies
 \begin{equation}\label{ct2}	
\forall y\in U(x,\rho), 0<\alpha<\rho: f(y+\alpha v(z))-f(y)\leq-2\eps \alpha.	
\end{equation}
 	By the choice of $\S_\eps^d$ there is some $\nu\in\S^d_\eps$ such that
	$|v(z)-\nu| < \eps$. By \eqref{ct2}, for each $y\in U(x,\rho)$ and $0<\alpha< \rho$, 
	$$f(y+\alpha v(z))-f(y)\leq-2\eps \alpha.$$
	Consequently, using $1$-Lipschitzness of $f\in X$, we obtain
	\begin{equation*}
	\begin{aligned}
	f(y+\alpha \nu)-f(y)&\leq f(y+\alpha v(z))-f(y)+|f(y+\alpha \nu)-f(y+\alpha v(z))|\\
	&\leq f(y+\alpha v(z))-f(y) + |\nu-v(z)|\alpha
	\leq -2\eps \alpha +\eps \alpha
	 = - \eps \alpha,
	\end{aligned}
	\end{equation*}
	and so  \eqref{eq:A_dFormula} holds.
\end{proof}

\begin{cor}\label{cor:AFsigmaDeltaSigma}
	The set $A$ is an $F_{\sigma\delta\sigma}$ subset of $X$.
\end{cor}

\begin{proof}
For each quadruple $y \in \er^d$,  $\nu \in  S^{d-1}$, $\alpha>0$, $\eps>0$ we set
$$ C(y,\nu,\alpha, \eps) \coloneqq  \{ f \in X:\ f(y+\alpha \nu)-f(y)\leq-\eps \alpha\}.$$
(Of course, we have  $C(y,\nu,\alpha, \eps)= \emptyset$ if $y \notin U(0,4)$ or
  $y+\alpha \nu\notin U(0,4)$.)  Further, for each triple $x\in U(0,4)$, $0<p<q$, we set
	$$ D(x,p,q)\coloneqq  \{ f \in X:\ f(x) \notin (p,q)\}.$$
	It is easy to see that both $C(y,\nu,\alpha, \eps)$ and $ D(x,p,q)$ are always closed subsets of
	 $X$.  It is easy to see that  
	 Lemma~\ref{lem:A_dFormula} is equivalent to
\begin{equation*}
\begin{aligned}
A=\bigcup_{ \eps \in\qp}\;
\bigcap_{\substack{p,q\in\qp,\\0<p<q<\eps}  } \;
\bigcup_{ \rho\in\qp} \;
\bigcap_{ x\in U(0,4)}\;
\left(D(x,p,q)\ \cup \
\bigcup_{ \nu\in\S^d_{\eps}}\ \bigcap_{\substack{y\in U(x,\rho),\\ 0<\alpha<\rho} } \ C(y,\nu,\alpha, \eps) \right).
\end{aligned}
\end{equation*}
	Therefore, since $\qp$ is countable and each $\S^d_{\eps}$ is finite, we obtain that
	   $A$ is an $F_{\sigma\delta\sigma}$ subset of $X$.
\end{proof}

{\bf The proof of Theorem \ref{complwdc}.} We know that it is sufficient to prove
 \eqref{staci2} for $r=1$. 

Suppose $d=2$. Then  Lemma \ref{lem:auraINXd}(ii) gives that $WDC(B(0,1))= \psi^{-1}(A \cap D)$,
 where $\psi: \K(B(0,4)) \to X$ is the continuous mapping from \eqref{spojpsi}. Since
 $A\cap D$ is an $F_{\sigma\delta\sigma}$ subset of $X$ by Corollary \ref{cor:AFsigmaDeltaSigma}
 and \eqref{bod},
  we obtain \eqref{staci2} for $r=1$ and $d=2$, and thus also assertion (i) of Theorem \ref{complwdc}.
	
To prove assertion (ii) of Theorem \ref{complwdc}, it is sufficient to prove  that (in each $\R^d$)
 $WDC(B(0,1))$ is an analytic subset of $\K(B(0,1))$.
To this end,  consider the following subset  $S$ of $\K(B(0,1)) \times X$:
$$ S\coloneqq  \{(K,f)\in \K(B(0,1)) \times X:\ f^{-1}(0)=K ,\ f\in A \cap D\}.$$
By   Lemma \ref{lem:auraINXd}(i),   $WDC(B(0,1)) = \pi_1(S)$ (where  $\pi_1(K,f)\coloneqq  K$)  and so it is sufficient to prove that $S$ is Borel.
Denoting $$Z\coloneqq   \{(K,f)\in \K(B(0,1)) \times X:\ K = f^{-1}(0),\ f\in X\},$$ we have $S= Z \cap (\K(B(0,1)) \times (A\cap D))$.
 So, since $A\cap D$ is Borel by Corollary \ref{cor:AFsigmaDeltaSigma}
 and \eqref{bod},  to prove that $S$ is Borel, it is sufficient to show that
 $Z$ is Borel in  $\K(B(0,1)) \times X$. To this end, denote for each $n\in \en$
$$ P_n\coloneqq  \{ (K,f) \in \K(B(0,1)) \times X:\ \exists x \in K: f(x) \geq 1/n\},$$
 $$Q_n\coloneqq  \{ (K,f) \in \K(B(0,1)) \times X:\ \exists x \in B(0,4): \ \dist(x,K)\geq 1/n, f(x)=0\}.$$
 Since clearly 
 $$Z=  (\K(B(0,1)) \times X) \setminus (\bigcup_{n=1}^{\infty} P_n \cup \bigcup_{n=1}^{\infty} Q_n),$$
 it is sufficient to prove that all $P_n$ and $Q_n$ are closed.

So suppose that $(K_i,f_i) \in \K(B(0,1)) \times X, \ i=1,2,\dots,$  $(K,f) \in \K(B(0,1)) \times X$, $\rho_H(K_i,K) \to 0$ and $\rho_{\sup}(f_i,f) \to 0$. 

First suppose that $n\in \en$ and all $(K_i,f_i) \in P_n$. Choose $x_i \in K_i$ with $f_i(x_i) \geq 1/n$. Choose a convergent subsequence $x_{i_j}\to x \in \R^d$. It is easy to see that $x \in K$.
Since $|f_{i_j}(x_{i_j}) - f(x_{i_j})| \to 0$  and  $f(x_{i_j}) \to f(x)$, we obtain
 $f_{i_j}(x_{i_j})\to f(x)$, and consequently $f(x)\geq 1/n$. Thus $(K,f) \in P_n$ and therefore
 $P_n$ is closed.

Second, suppose that $n\in \en$ and all $(K_i,f_i) \in P_n$. Choose $x_i \in B(0,4)$ such that
 $\dist(x_i,K_i)\geq 1/n$ and  $f_i(x_i)=0$. Choose a convergent subsequence $x_{i_j}\to x \in B(0,4)$. Since $|f_{i_j}(x_{i_j}) - f(x_{i_j})| \to 0$  and  $f(x_{i_j}) \to f(x)$, we obtain $f(x)=0$.
 Now consider an arbitrary $y \in K$ and choose a sequence $y_j \in K_{i_j}$ with $y_j \to y$.
 Since $|x_{i_j} - y_j| \geq 1/n$ and $x_{i_j}\to x$, we obtain that $|y-x|\geq 1/n$ 
 and consequently $\dist(x,K) \geq 1/n$. Thus $(K,f) \in Q_n$ and therefore
 $Q_n$ is closed.

\end{document}